\documentclass{amsart}
\usepackage{amsmath,amsfonts,amsthm}
\usepackage{amssymb,latexsym}
\usepackage{graphics}
\usepackage[colorlinks]{hyperref}

\usepackage[all]{xypic}

\theoremstyle{plain}
\theoremstyle{definition}
\newtheorem{theorem}{Theorem}[section]
\newtheorem{lemma}[theorem]{Lemma}

\newtheorem{corollary}[theorem]{Corollary}

\newtheorem{note}[theorem]{Note}

\newtheorem{convention}[theorem]{Convention}
\newtheorem{remark}[theorem]{Remark}
\theoremstyle{remark}

\numberwithin{equation}{section}

\title{Distributional chaotic generalized shifts}

\author[Z. Nili Ahmadabadi, F. Ayatollah Zadeh Shirazi]{Zahra Nili Ahmadabadi, Fatemah Ayatollah Zadeh Shirazi}

\begin{document}
 \maketitle

\begin{abstract}
Suppose $X$ is a finite discrete space with at least two elements, $\Gamma$ is a nonempty countable set, and consider self--map
$\varphi:\Gamma\to\Gamma$. We prove that the generalized shift $\sigma_\varphi:X^\Gamma\to X^\Gamma$ with
$\sigma_\varphi((x_\alpha)_{\alpha\in\Gamma})=(x_{\varphi(\alpha)})_{\alpha\in\Gamma}$ (for
$(x_\alpha)_{\alpha\in\Gamma}\in X^\Gamma$) is:
\begin{itemize}
\item distributional chaotic (uniform, type 1, type 2) if and only if
	$\varphi:\Gamma\to\Gamma$ has at least a non-quasi-periodic point,
\item dense distributional chaotic if and only if
	$\varphi:\Gamma\to\Gamma$ does not have any periodic point,
\item transitive distributional chaotic if and only if
	$\varphi:\Gamma\to\Gamma$ is one--to--one without any periodic point.
\end{itemize}
We complete the text by counterexamples.
\end{abstract}

\

\noindent {\bf AMS Classification:} {54H20 }
\\
{\bf Keywords:} {Distributional chaotic, Generalized shift, Li-Yorke chaotic. } \vspace{1cm} \maketitle

\section{Introduction}
\noindent ``CHAOS'' is one of the most famous mathematical terms and concepts that has been studied in
dynamical systems. Regarding these studies one may consider different kinds of chaos like
Devaney chaos~\cite{devaneychaos2, devaneychaos}, Li--Yorke chaos~\cite{liyorkechaos},
topological chaos (i.e., nonzero topological entropy~\cite{walters}),
distributional chaos  etc.. One can study either general properties of chaotic maps
(apart from phase space or phase map), or special cases.
In this paper we deal with different types of distributional chaos in generalized shift dynamical systems.
In next section we bring a collection of preliminaries, in Section 3 we study
distributional chaotic generalized shifts of type 1 (and 2). In Sections 4 and 5 we study dense
distributional chaotic generalized shifts and transitive distributional chaotic generalized shifts.
Finally Section 6 is covered by counterexamples.
\section{Preliminaries}
\noindent By a topological dynamical system $(Y,g)$ we mean a compact metric space $Y$ and continuous map $g:Y\to Y$.
In dynamical system $(Y,g)$ we say $y\in Y$ is a transitive point if $\{g^n(y):n\geq0\}$ is a dense subset of $Y$ and in this case we say
$(Y,g)$ is point transitive. Moreover we say
$(Y,g)$ is topological transitive (or simply transitive) if for all opene (nonempty and open) subsets $U,V$ of $Y$ there exists $n\geq1$ with $g^n(U)\cap V\neq\varnothing$.
It's well-known that for compact perfect metric space $Y$, $(Y,g)$ is topological transitive if and only if it is point transitive
(\cite[Proposition 1.1]{silverman} and \cite{kolyada}). In dynamical system $(Y,g)$ we say nonempty subset $Z$ of $Y$
is invariant (or $g-$invariant) if $g(Z)\subseteq Z$.
\subsection*{Background on generalized shifts}
One-sided and two-sided shifts (resp.
\linebreak
$\mathop{\{1,\ldots,k\}^{\mathbb N}\to \{1,\ldots,k\}^{\mathbb N}}\limits_{(x_n)_{n\in{\mathbb N}}\mapsto(x_{n+1})_{n\in{\mathbb N}}}$ and
$\mathop{\{1,\ldots,k\}^{\mathbb Z}\to \{1,\ldots,k\}^{\mathbb Z}}\limits_{(x_n)_{n\in{\mathbb Z}}\mapsto(x_{n+1})_{n\in{\mathbb Z}}}$)
are one of the most applicable tools in dynamical systems and ergodic theory \cite{walters}. Generalized shifts
have been introduced for the first time in \cite{AHK}, where for nonempty set $\Lambda$, arbitrary set $Y$ with
at least two elements and self-map $\theta:\Lambda\to\Lambda$, we call $\sigma_\theta:Y^\Lambda\to Y^\Lambda$ with
$\sigma_\theta((x_\alpha)_{\alpha\in\Lambda})=(x_{\theta(\alpha)})_{\alpha\in\Lambda}$ (for
$(x_\alpha)_{\alpha\in\Lambda}\in Y^\Lambda$) a generalized shift. It's evident that if $Y$ is a topological space and
$Y^\Lambda$  equipped with product topology, then  $\sigma_\theta:Y^\Lambda\to Y^\Lambda$
is continuous, moreover if $Y$ has a group structure, then
$\sigma_\theta:Y^\Lambda\to Y^\Lambda$ is a group homomorphism too, so one may study both dynamical
and non-dynamical properties of generalized shifts \cite{dev,gio}.
\begin{convention} 
Henceforth suppose $X$ is a finite discrete set with at least two elements,
$\Gamma$ is a nonempty countable set, and $\varphi:\Gamma\to\Gamma$ is arbitrary. Consider $X^\Gamma$ with
product (pointwise convergence) topology and for $D\subseteq \Gamma$ let:
\[\gamma_D:=\{((x_\alpha)_{\alpha\in\Gamma},(y_\alpha)_{\alpha\in\Gamma})
\in X^\Gamma\times X^\Gamma:\forall\alpha\in D \:(x_\alpha=y_\alpha)\}\: .\]
\\
Under the above assumptions $X^\Gamma$ is compact Hausdorff, so it is unifomizable,
moreover $X^\Gamma$ is countable product of metrizable spaces, hence it is metrizable,
suppose $d$ is a compatible metric on $X^\Gamma$, for $\varepsilon>0$ let
\[\alpha_\varepsilon:=
\{(x,y)\in X^\Gamma\times X^\Gamma:d(x,y)<\varepsilon\}\:.\]
Both of the following sets are
compatible uniform structures on $X^\Gamma$:
\[\mathcal{F}_p=\{B\subseteq X^\Gamma\times X^\Gamma:\textrm{ there exists finite subset } D \textrm{ of }\Gamma \textrm{ with }\gamma_D\subseteq B\}\:,\]
\[\mathcal{F}_d=\{B\subseteq X^\Gamma\times X^\Gamma:\exists\varepsilon>0\:(\alpha_\varepsilon\subseteq B)\}\:.\]
Note to the fact that every compact Hausdorff space admits a unique compatible uniform structure, we have $\mathcal{F}_p=\mathcal{F}_d=:\mathcal{F}$.
For more details on uniform spaces see~\cite{dugundji, engelking}.
\end{convention}
\subsection*{Background on distributional chaotic dynamical systems}
In compact metric space $(Y,d)$ for continuous map $f:Y\to Y$,
$x,y\in Y$, $n\in\mathbb{N}$ and $t\in \mathbb R$ let (by $\#A$ we mean cardinality of $A$ for finite $A$ and $+\infty$ for infinite $A$):
\[\xi(x,y,t,n)=\#\left\{i\in\{0,\ldots,n-1\}:d(f^i(x),f^i(y))<t\right\}\]
and
\[F_{xy}(t)={\displaystyle\liminf_{n\to\infty}\frac{\xi(x,y,t,n)}{n}}\:,\:
F^*_{xy}(t)={\displaystyle\limsup_{n\to\infty}\frac{\xi(x,y,t,n)}{n}\:.}\]
We say $x,y\in Y$ are distributional scrambled of type 1, if
\begin{itemize}
\item $\exists s>0\: \: \: \: \:(F_{xy}(s)=0)$,
\item $\forall s>0\: \: \: \: \:(F^*_{xy}(s)=1)$.
\end{itemize}
We say $x,y\in Y$ are distributional scrambled of type 2, if
\begin{itemize}
\item $\exists s>0\: \: \: \: \:(F_{xy}(s)<1)$,
\item $\forall s>0\: \: \: \: \:(F^*_{xy}(s)=1)$.
\end{itemize}
We say $x,y\in Y$ are distributional scrambled of type 3, if there exist $b>a>0$ with
$F_{xy}(s)<F^*_{xy}(s)$ for all $s\in[a,b]$.
\\
We say $A (\subseteq Y)$  with at least two elements is a distributional scrambled subset of type $i$
of $Y$ if for all distinct $x,y\in A$, $x,y$ are distributional scrambled of type $i$. Also we say
$f:Y\to Y$ is DC$^ui$ (resp. DC$^\infty i$, DC$^2i$)
if $Y$ has an uncountable (resp. infinite, with at least two elements)
distributional scrambled subset of type $i$.
\\
If $(Y,f)$ is DC$^u$1 and $Y$ has an uncountable distributional
chaotic set of type 1 like $A$ such that there exists $\varepsilon>0$ with $F_{xy}(\varepsilon)=0$
for all distinct $x,y\in A$, then we say $(Y,f)$ is uniform distributional chaotic.
On the other hand if $(Y,f)$ is uniform distributional chaotic and $A(\subseteq Y)$ is dense
and satisfies all of the above conditions, then we say
$(Y,f)$ is dense distributional chaotic. If $(Y,f)$ is  uniform distributional chaotic
$A(\subseteq Y)$ is dense consisting of transitive points
and satisfies all of the above conditions, then we say
$(Y,f)$ is transitive distributional chaotic (see \cite{revis} too).
\begin{note}\label{salam10}
For $h\in\{u,\infty,2\}$, it is clear that if
$(Y,f)$ is DC$^h$1, then it is DC$^h$2. However using the fact that for all
$x,y\in Y$ and $s,t>0$ with $t<s$ we have $F_{xy}(t)\leq F_{xy}(s)$, so if $(Y,f)$
is DC$^h$2, then it is DC$^h$3.
\end{note}
\subsection*{Background on Li-Yorke chaotic dynamical systems}
In compact metric space $(Y,d)$ for continuous map $f:Y\to Y$ we say $x,y\in Y$
are Li-Yorke scrambled if $\mathop{\liminf}\limits_{n\to\infty}d(f^n(x),f^n(y))=0$ and
$\mathop{\limsup}\limits_{n\to\infty}d(f^n(x),f^n(y))>0$.
We say $A (\subseteq Y)$  with at least two elements is a Li-Yorke scrambled subset
of $Y$ if for all distinct $x,y\in A$, $x,y$ are Li-Yorke scrambled. Also we say
$f:Y\to Y$ is LY$_u$ chaotic (resp. LY$_{\infty}$ chaotic, LY$_2$ chaotic if $Y$ has
an uncountable (resp. infinite, with at least two elements) Li-Yorke scrambled subset \cite{khoob}.
\begin{remark}\label{salam20}
The generalized shift dynamical system $(X^\Gamma,\sigma_\varphi)$ is
 LY$_u$ (resp.  LY$_{\infty}$, LY$_2$) chaotic if and only if $\varphi:\Gamma\to\Gamma$
 has at least a non-quasi-periodic point \cite{yorke}.
\end{remark}
\section{DC1 and DC2 generalized shifts}
\noindent In this section we prove that the generalized shift dynamical system $(X^\Gamma,\sigma_\varphi)$ is
uniform distributional chaotic (resp. DC$^u$1, DC$^\infty$1, DC$^2$1, DC$^u$2, DC$^\infty$2, DC$^2$2) if and only if
$\varphi:\Gamma\to\Gamma$ has at least a non-quasi-periodic point.
\begin{lemma}\label{salam35}
For $x,y\in X^\Gamma$, $n\geq1$, $\alpha\in{\mathcal F}$ and $f:X^\Gamma\to X^\Gamma$ let:
\[\zeta(x,y,\alpha,n)=\#\left\{i\in\{0,\ldots,n-1\}:(f^i(x),f^i(y))\in\alpha\right\}\]
and
\[G_{xy}(\alpha)={\displaystyle\liminf_{n\to\infty}\frac{\zeta(x,y,\alpha,n)}{n}}\:,\:
G^*_{xy}(\alpha)={\displaystyle\limsup_{n\to\infty}\frac{\zeta(x,y,\alpha,n)}{n}\:.}\]
Now we have:
\begin{itemize}
\item the following statements are equivalent:
	\begin{itemize}
	\item[1.] there exists $s>0$ with $F_{xy}(s)=0$,
	\item[2.] there exists $\alpha\in{\mathcal F}$ with $G_{xy}(\alpha)=0$,
	\item[3.] there exists finite subset $D$ of $\Gamma$ with $G_{xy}(\gamma_D)=0$,
	\end{itemize}
\item the following statements are equivalent:
	\begin{itemize}
	\item[4.] there exists $s>0$ with $F_{xy}(s)<1$,
	\item[5.] there exists $\alpha\in{\mathcal F}$ with $G_{xy}(\alpha)<1$,
	\item[6.] there exists finite subset $D$ of $\Gamma$ with $G_{xy}(\gamma_D)<1$,
	\end{itemize}
\item the following statements are equivalent:
	\begin{itemize}
	\item[7.] for all $s>0$ we have $F^*_{xy}(s)=1$,
	\item[8.] for all $\alpha\in{\mathcal F}$ we have $G^*_{xy}(\alpha)=1$,
	\item[9.] for all finite subset $D$ of $\Gamma$ we have $G^*_{xy}(\gamma_D)=1$.
	\end{itemize}
\end{itemize}
\end{lemma}
\begin{proof}
Note that for $s>0$ and $n\geq1$ we have $\xi(x,y,s,n)=\zeta(x,y,\alpha_s,n)$, so
$F_{xy}(s)=G_{xy}(\alpha_s)$ (and $F^*_{xy}(s)=G^*_{xy}(\alpha_s)$). Therefore ``$1\Rightarrow2$'' (and
``$4\Rightarrow5$'').
\vspace{3mm}
\\
By ${\mathcal F}_p=\mathcal{F}$, for each $\alpha\in\mathcal{F}$ there exists finite subset $D$ of $\Gamma$
with $\gamma_D\subseteq\alpha$, so $G_{xy}(\gamma_D)\leq G_{xy}(\alpha)$. Therefore ``$2\Rightarrow3$'' (and
``$5\Rightarrow6$'').
\vspace{3mm}
\\
By ${\mathcal F}_p=\mathcal{F}_d$, for each finite subset $D$ of $\Gamma$ there exists $s>0$ with $\alpha_s\subseteq\gamma_D$, so
$F_{xy}(s)=G_{xy}(\alpha_s)\leq G_{xy}(\gamma_D)$. Therefore ``$3\Rightarrow1$'' (and
``$6\Rightarrow4$'').
\vspace{3mm}
\\
``$7\Rightarrow8$'' Suppose for all $s>0$, $F^*_{xy}(s)=1$. For all
$\alpha\in\mathcal{F}$ there exists $s>0$ with $\alpha_s\subseteq\alpha$,
so $1=F^*_{xy}(s)=G^*_{xy}(\alpha_s)\leq G^*_{xy}(\alpha)\leq1$
which leads to $G^*_{xy}(\alpha)=1$.
\vspace{3mm}
\\
It's evident that ``$8\Rightarrow9$''.
\vspace{3mm}
\\
``$9\Rightarrow7$''
for all finite subset $D$ of $\Gamma$ we have $G^*_{xy}(\gamma_D)=1$.
For $s>0$ there exists finite subset $D$ of $\Gamma$ such that
$\gamma_D\subseteq\alpha_s$, so $1=G^*_{xy}(\gamma_D)\leq G^*_{xy}(\alpha_s)=F^*_{xy}(s)\leq1$ which leads to $F^*_{xy}(s)=1$.
\end{proof}
\begin{corollary}\label{salam37}
Using a similar method described in Lemma~\ref{salam35},  $f:X^\Gamma\to X^\Gamma$
is uniform distributional chaotic if and only if it is DC$^u$1 and $X^\Gamma$ has an uncountable distributional
chaotic set of type 1 like $A$ such that for all distinct $x,y\in A$ there exists $\alpha\in{\mathcal F}$ (resp. finite subset $D$ of $\Gamma$) with $G_{xy}(\alpha)=0$
(resp. $G_{xy}(\gamma_D)=0$).
\end{corollary}
\begin{remark}\label{salam40}
There exists uncountable family $\mathcal K$ of infinite subsets of $\mathbb N$ such that
for all $A,B\in\mathcal{K}$, $A\cap B$ is finite \cite{size, yorke}.
\end{remark}
\begin{lemma}\label{salam50}
If $\varphi:\Gamma\to\Gamma$ has a non-quasi-periodic point, then
$(X^\Gamma,\sigma_\varphi)$ is uniform distributional chaotic.
\end{lemma}
\begin{proof}
Suppose $\theta\in\Gamma$ is a non-quasi-periodic point of $\varphi$.
Choose an strictly increasing sequence $\{s_n\}_{n\geq1}$ of natural numbers
such that for all $n\geq1$ we have $\dfrac{s_n}{s_1+\cdots+s_n}>\dfrac{n-1}{n}$.
Choose distinct $p,q\in X$. For $A\subseteq \mathbb N$ let:
\[z_i^A=\left\{\begin{array}{lc} p & i\in A\: , \\
	q & \textrm{ otherwise\:}. \end{array}\right.\]
Now suppose $x_\beta^A=q$ for $\beta\in\Gamma\setminus\{\varphi^n(\theta):n\geq0\}$ and:
\[(x_\theta^A,x_{\varphi(\theta)}^A,x_{\varphi^2(\theta)}^A,\cdots)=
(\underbrace{z_1^A,\cdots,z_1^A}_{s_1\textrm{ times}},
\underbrace{z_2^A,\cdots,z_2^A}_{s_2\textrm{ times}},
\underbrace{z_3^A,\cdots,z_3^A}_{s_3\textrm{ times}},\cdots) \]
Also let $x^A:=(x_\alpha^A)_{\alpha\in\Gamma}$. We continue the proof through two claims.
\vspace{3mm}
\\
{\bf Claim 1.} For $A,B\subseteq {\mathbb N}$ with infinite $A\cap B$
we have $F_{x^Ax^B}^*(t)=1$ for all $t>0$.
\vspace{3mm}
\\
{\it Proof of Claim 1}. 
Consider finite subsets 
\[D\subseteq \Gamma\setminus\bigcup\{\varphi^i(\theta):
i\in{\mathbb Z}\}\:\:,\:\: E\subseteq \{\varphi^i(\theta):i\geq0\}\:\:,\:\:
F\subseteq \bigcup\{\varphi^{-i}(\theta):i\geq1\}\:.\]
It's evident that: 
\[\forall n\geq0\: \: \: \: \: (\sigma_\varphi^n(x^A),\sigma_\varphi^n(x^B))\in\gamma_D\:.\tag{A}\]
There exists
$N\geq1$ with 
$E\subseteq\{\varphi^i(\theta):0\leq i\leq N\}$ and $F\subseteq\bigcup\{\varphi^{-i}(\theta):1\leq i\leq N\}$.
For all 
$r\in A\cap B
\setminus\{0,1,\ldots,2N\}$ we have $s_r>2N$ and
\[3N-1< r-1+N\leq s_1+\cdots+s_{r-1}+N<s_1+\cdots+s_r-N\:,\]
so for all $j,k$ with $s_1+\cdots+s_{r-1}+N<j<s_1+\cdots+s_r-N$ and $0\leq k\leq N$ we have
$0\leq s_1+\cdots+s_{r-1}+N-k<j<s_1+\cdots+s_r+N-k$ and
$s_1+\cdots+s_{r-1}<j+k-N<s_1+\cdots+s_r$ so
\[x^A_{\varphi^{(j+k-N)}(\theta)}=p=x^B_{\varphi^{(j+k-N)}(\theta)}\:.\tag{*}\]
For $\alpha\in E$
there exists $0\leq k\leq N$ with $\varphi^k(\theta)=\alpha$,
thus by $(*)$ for all $s_1+\cdots+s_{r-1}+N<j<s_1+\cdots+s_r-N$ we have:
\[x^A_{\varphi^{(j-N)}(\alpha)}=p=x^B_{\varphi^{(j-N)}(\alpha)}\:.\tag{B}\]
For $\alpha\in F$
there exists $1\leq l\leq N$ with $\varphi^l(\alpha)=\theta$,
thus by $(*)$ for all $0\leq k\leq N$ and $s_1+\cdots+s_{r-1}+N<j<s_1+\cdots+s_r-N$ we have:
\[x^A_{\varphi^{(j+k-N)}(\varphi^l(\alpha))}=p=x^B_{\varphi^{(j+k-N)}(\varphi^l(\alpha))}\]
let $k=N-l$, hence:
\[x^A_{\varphi^{(j+N)}(\alpha)}=p=x^B_{\varphi^{(j+N)}(\alpha)}\:.\tag{C}\]
Therefore:
\\
$\#\{i\in\{0,\ldots,s_1+\cdots+s_r-1\}:(\sigma_\varphi^i(x^A),\sigma_\varphi^i(x^B))
	\in \gamma_{D\cup E\cup F}\}$
\begin{eqnarray*}
	& \mathop{=}\limits^{(\textrm{A})} & \#\{i\in\{0,\ldots,s_1+\cdots+s_r-1\}:
	(\sigma_\varphi^i(x^A),\sigma_\varphi^i(x^B))\in \gamma_{E\cup F}\} \\
	& = & \#\{i\in\{0,\ldots,s_1+\cdots+s_r-1\}:\forall\alpha\in E\cup F\:\:\:
	x_{\varphi^i(\alpha)}^A=x_{\varphi^i(\alpha)}^B\} \\
& \mathop{\geq}\limits^{(\textrm{(C)  and  (D)})} & \#\{s_1+\cdots+s_{r-1}+2N+1,\ldots,s_1+\cdots+s_r-2N-1\} \\
& = & s_r-4N-1\:.
\end{eqnarray*}
Hence:
\vspace{3mm}
\\
$G^*_{x^Ax^B}(\gamma_{D\cup E\cup F})$
\begin{eqnarray*}
& = & \mathop{\limsup}\limits_{n\to\infty}\dfrac{\#\{i\in\{0,\ldots,n-1\}:(\sigma_\varphi^i(x^A),		
	\sigma_\varphi^i(x^B))\in \gamma_{D\cup E\cup F}\}}{n} \\
& \geq & \mathop{\limsup}\limits_{r\to\infty, r\in A\cap B}\dfrac{\#\{i\in\{0,\ldots,s_1+\cdots+s_r-1\}:(\sigma_\varphi^i(x^A),		
	\sigma_\varphi^i(x^B))\in \gamma_{D\cup E\cup F}\}}{s_1+\cdots+s_r} \\
& \geq &  \mathop{\limsup}\limits_{r\to\infty, r\in A\cap B}\dfrac{s_r-4N-1}{s_1+\cdots+s_r} \\
& \geq &  \mathop{\limsup}\limits_{r\to\infty, r\in A\cap B}(\dfrac{r-1}{r}-\dfrac{4N+1}{s_1+\cdots+s_r})=1
\end{eqnarray*}
So $G^*_{x^Ax^B}(\gamma_{D\cup E\cup F})=1$ which leads to
$G^*_{x^Ax^B}(\gamma_{M})=1$ for all finite subset $M$ of $\Gamma$.
Thus by Lemma~\ref{salam35}, $F^*_{x^Ax^B}(t)=1$ for all $t>0$.
\vspace{3mm}
\\
{\bf Claim 2.} There exists $t>0$ such that for all $A,B\subseteq {\mathbb N}$
with infinite $A\setminus B$ we have $F_{x^Ax^B}(t)=0$.
\vspace{3mm}
\\
{\it Proof of Claim 2}. There exists $t>0$ with
$\{(z,w)\in X^\Gamma\times X^\Gamma:d(z,w)<t\}\subseteq\gamma_{\{\theta\}}$.
For $r\in A\setminus B$
and $s_1+\cdots+s_{r-1}<i<s_1+\cdots+s_r$ we have $x_{\varphi^i(\theta)}^A=p$ and
$x_{\varphi^i(\theta)}^B=q$, so $x_{\varphi^i(\theta)}^A\neq x_{\varphi^i(\theta)}^B$
and $(\sigma_\varphi^i(x^A),\sigma_\varphi^i(x^B))\notin\gamma_{\{\theta\}}$. Hence
we have:
\begin{eqnarray*}
F_{x^Ax^B}(t) & = & \mathop{\liminf}\limits_{n\to\infty}\dfrac{\#\{i\in\{0,\ldots,n-1\}:d(\sigma_\varphi^i(x^A),		
	\sigma_\varphi^i(x^B))<t\}}{n} \\
& \leq & \mathop{\liminf}\limits_{n\to\infty}
	\dfrac{\#\{i\in\{0,\ldots,n-1\}:(\sigma_\varphi^i(x^A),		
	\sigma_\varphi^i(x^B))\in\gamma_{\{\theta\}}\}}{n} \\
& \leq & \mathop{\liminf}\limits_{r\to\infty, r\in A\setminus B}
	\dfrac{\#\{i\in\{0,\ldots,s_1+\cdots+s_r-1\}:(\sigma_\varphi^i(x^A),		
	\sigma_\varphi^i(x^B))\in\gamma_{\{\theta\}}\}}{s_1+\cdots+s_r} \\
& \leq & \mathop{\liminf}\limits_{r\to\infty, r\in A\setminus B}
	\dfrac{s_1+\cdots+s_{r-1}}{s_1+\cdots+s_r} \\
& = & 1-\mathop{\limsup}\limits_{r\to\infty, r\in A\setminus B}
	\dfrac{s_r}{s_1+\cdots+s_r} \\
& \leq & 1-\mathop{\limsup}\limits_{r\to\infty, r\in A\setminus B}
	\dfrac{r-1}{r}=0
\end{eqnarray*}
which leads to $F_{x^Ax^B}(t)=0$.
\vspace{3mm}
\\
Now we are ready to prove Lemma, by Remark~\ref{salam40} there exists uncountable subset
$\mathcal{K}$ of infinite subsets of ${\mathbb N}\setminus 2{\mathbb N}$
such that for all $A,B\in\mathcal{K}$, $A\cap B$ is finite, thus $A\setminus B$ is infinite.
Let $\mathcal{H}:=\{A\cup2{\mathbb N}:A\in\mathcal{K}\}$, then for all distinct $C,D\in
\mathcal{H}$, both sets $C\cap D,C\setminus D$ are infinite. Hence by Claims 1 and 2
there exists $t>0$ such that for all distinct $C,D\in\mathcal{H}$ we have $F^*_{x^Cx^D}=1$
and $F_{x^Cx^D}(t)=0$. Therefore $\{x^C:C\in\mathcal{H}\}$ is an uncountable distributional
scrambled set of type 1. In particular, $(X^\Gamma,\sigma_\varphi)$ is uniform
distributional chaotic.
\end{proof}
\begin{theorem}\label{salam60}
The generalized shift dynamical system $(X^\Gamma,\sigma_\varphi)$ is
uniform distributional chaotic
(resp. DC$^u$1, DC$^\infty$1, DC$^2$1, DC$^u$2, DC$^\infty$2, DC$^2$2) if and only if
$\varphi:\Gamma\to\Gamma$ has at least a non-quasi-periodic point.
\end{theorem}
\begin{proof}
If $(X^\Gamma,\sigma_\varphi)$ is uniform distributional chaotic
(resp. DC$^u$1, DC$^\infty$1, DC$^2$1, DC$^u$2, DC$^\infty$2), then it is DC$^2$2
and has a distributional scrambled pair of type 2, this pair is a Li-Yorke scrambled pair too
and $(X^\Gamma,\sigma_\varphi)$ is LY$^2$ and by Remark~\ref{salam20}
\linebreak
$\varphi:\Gamma\to\Gamma$
 has at least a non-quasi-periodic point. On the other hand if 
 \linebreak
 $\varphi:\Gamma\to\Gamma$
 has  a non-quasi-periodic point, then by Lemma~\ref{salam50}, $(X^\Gamma,\sigma_\varphi)$ is uniform distributional chaotic and hence DC$^u$1,
 use Note~\ref{salam10} to complete the proof.
\end{proof}
\noindent
Let's mention that distributional chaos type 3, depends on chosen compatible metric of phase space
\cite[Theorem 2]{3ver}. Now we have the following question:
\\
{\bf Problem.} Suppose $\Gamma=\{\beta_1,\beta_2,\ldots\}$ with distinct $\beta_i$s and
equip $X^\Gamma$ with compatible metric
\[D((x_\alpha)_{\alpha\in\Gamma},(y_\alpha)_{\alpha\in\Gamma})={\displaystyle\sum_{n\geq1}
\dfrac{\delta(x_{\beta_i},y_{\beta_i})}{2^i}}\:\:((x_\alpha)_{\alpha\in\Gamma},(y_\alpha)_{\alpha\in\Gamma}\in X^\Gamma)\]
where
$\delta(a,b)=0$ for $a\neq b$ and $\delta(a,a)=1$. Under which conditions
$(X^\Gamma,\sigma_\varphi)$ is DC$^u$3 (DC$^\infty$3, DC$^2$3)?
\section{Dense distributional chaotic generalized shifts}
\noindent In this section we see $(X^\Gamma,\sigma_\varphi)$ is dense distributional chaotic if and only if
$\varphi:\Gamma\to\Gamma$ does not have any periodic point.
\begin{lemma}\label{salam70}
If $(X^\Gamma,\sigma_\varphi)$ is dense distributional chaotic, then
$\varphi:\Gamma\to\Gamma$ does not have any periodic point.
\end{lemma}
\begin{proof}
Choose distinct $p,q\in X$, if $\theta$ is a periodic point of $\varphi$, then $\{\varphi^i(\theta):i\geq0\}$
is a finite subset of $\Gamma$ and there exists $t>0$ such that
$\{(x,y)\in X^\Gamma\times X^\Gamma:d(x,y)<t\}\subseteq \gamma_{\{\varphi^i(\theta):i\geq0\}}$.
If $A$ is a dense distributional scrambled subset of $X^\Gamma$ of type 1, then for open subsets
$U=\mathop{\prod}\limits_{\alpha\in\Gamma}U_\alpha$ 
and $V=\mathop{\prod}\limits_{\alpha\in\Gamma}V_\alpha$ of $X^\Gamma$
with:
\[U_\alpha=\left\{\begin{array}{lc} \{p\} & \alpha\in\{\varphi^n(\theta):n\geq0\}\:, \\
X & \alpha\notin\{\varphi^n(\theta):n\geq0\}\:, \end{array} \right. \: \: \: \: \: 
V_\alpha=\left\{\begin{array}{lc} \{q\} & \alpha\in\{\varphi^n(\theta):n\geq0\}\:, \\
X & \alpha\notin\{\varphi^n(\theta):n\geq0\}\:, \end{array} \right. \]
then there exist
$x=(x_\alpha)_{\alpha\in\Gamma}\in A\cap U$ and $y=(y_\alpha)_{\alpha\in\Gamma}\in A\cap V$ so
\linebreak
$\{x_{\varphi^n(\theta)}:n\geq0\}=\{p\}$ and $\{y_{\varphi^n(\theta)}:n\geq0\}=\{q\}$,
hence for all $i\geq0$ we have
$(\sigma_\varphi^i(x),\sigma_\varphi^i(y))\notin\gamma_{\{\varphi^i(\theta):i\geq0\}}$
and $d(\sigma_\varphi^i(x),\sigma_\varphi^i(y))\geq t$, thus
\[\#\{i\in\{0,\ldots,n-1\}:d(\sigma_\varphi^i(x),\sigma_\varphi^i(y))<t\}=0\]
and $F^*_{xy}(t)=F_{xy}(t)=0$ which is a contradiction, hence $X^\Gamma$
does not have any dense distributional scrambled subset of type 1, thus
$(X^\Gamma,\sigma_\varphi)$ is not dense distributional chaotic.
\end{proof}
\begin{lemma}\label{salam75}
For $s>0$, there exists $r>0$ such that
for all
$x=(x_\alpha)_{\alpha\in\Gamma}, y=(y_\alpha)_{\alpha\in\Gamma},z=(z_\alpha)_{\alpha\in\Gamma}
\in X^\Gamma$ if $\{\alpha\in\Gamma:x_\alpha\neq z_\alpha\}$ is a finite collection of non--quasi periodic points of $\varphi$, then
$F_{xy}^*(r)\leq F_{zy}^*(s)$ and $F_{xy}(r)\leq F_{zy}(s)$. 
\end{lemma}
\begin{proof}
Choose $s>0$, there exist
$\beta_1,\ldots,\beta_m\in\Gamma$ and $r>0$ with
{\small\[\{(u,v)\in X^\Gamma\times X^\Gamma:d(u,v)<r\}\subseteq
\gamma_{\{\beta_1,\ldots,\beta_m\}}\subseteq
\{(u,v)\in X^\Gamma\times X^\Gamma:d(u,v)<s\}\:.\tag{*}\]}
For
$x=(x_\alpha)_{\alpha\in\Gamma}, y=(y_\alpha)_{\alpha\in\Gamma},z=(z_\alpha)_{\alpha\in\Gamma}
\in X^\Gamma$ suppose 
$\{\alpha\in\Gamma:x_\alpha\neq z_\alpha\}$ is a finite collection of non--quasi periodic points of $\varphi$, then
there exists $N\geq1$ such that
\[\{\alpha\in\Gamma:x_\alpha\neq z_\alpha\}\cap\{\varphi^i(\{\beta_1,\ldots,\beta_m\}):i\geq N\}=\varnothing\:.\]
Now we have:

$\{\alpha\in\Gamma:x_\alpha\neq z_\alpha\}\cap\{\varphi^i(\{\beta_1,\ldots,\beta_m\}):i\geq N\}=\varnothing$
\begin{eqnarray*}
& \Rightarrow & \forall i\geq N\:\forall j\in\{1,\ldots ,m\}
	\: \: \: \: \:(x_{\varphi^i(\beta_j)}=z_{\varphi^i(\beta_j)}) \\
& \Rightarrow & \forall i\geq N\:\forall j\in\{1,\ldots ,m\}
	\: \: \: \: \:((x_{\varphi^i(\beta_j)}=y_{\varphi^i(\beta_j)})\Leftrightarrow
	(z_{\varphi^i(\beta_j)}=y_{\varphi^i(\beta_j)})) \\
& \Rightarrow & \forall i\geq N\: \: \: \: \:(
	(\sigma_\varphi^i(x),\sigma_\varphi^i(y))\in
	\gamma_{\{\beta_1,\ldots,\beta_m\}}\Leftrightarrow
	(\sigma_\varphi^i(z),\sigma_\varphi^i(y))\in
	\gamma_{\{\beta_1,\ldots,\beta_m\}})
\end{eqnarray*}
Hence for all $n\geq1$ we have
\\
$(**)\: \: \: \: \: \left|
\#\left\{i\in\{0,\ldots,n-1\}:(\sigma_\varphi^i(x),\sigma_\varphi^i(y))\in
	\gamma_{\{\beta_1,\ldots,\beta_m\}}\right\}\right.-$
	\[\left.\#\left\{i\in\{0,\ldots,n-1\}:(\sigma_\varphi^i(z),\sigma_\varphi^i(y))\in
	\gamma_{\{\beta_1,\ldots,\beta_m\}}\right\}\right|\leq N\]
Using (*) and (**) we have

$\#\left\{i\in\{0,\ldots,n-1\}:d(\sigma_\varphi^i(x),\sigma_\varphi^i(y))<r\right\}$
\begin{eqnarray*}
	& \leq &
	\#\left\{i\in\{0,\ldots,n-1\}:(\sigma_\varphi^i(x),\sigma_\varphi^i(y))\in
	\gamma_{\{\beta_1,\ldots,\beta_m\}}\right\} \\
& \leq & \#\left\{i\in\{0,\ldots,n-1\}:(\sigma_\varphi^i(z),\sigma_\varphi^i(y))\in
	\gamma_{\{\beta_1,\ldots,\beta_m\}}\right\} +N \\
& \leq &
	\#\left\{i\in\{0,\ldots,n-1\}:d(\sigma_\varphi^i(z),\sigma_\varphi^i(y))<s\right\}+N
\end{eqnarray*}
Therefore:
\begin{eqnarray*}
F^*_{xy}(r) & = & {\displaystyle\limsup_{n\to\infty}
	\dfrac{\#\left\{i\in\{0,\ldots,n-1\}:d(\sigma_\varphi^i(x),\sigma_\varphi^i(y))
	<r\right\}}{n}} \\
& \leq & {\displaystyle\limsup_{n\to\infty}
	\dfrac{\#\left\{i\in\{0,\ldots,n-1\}:d(\sigma_\varphi^i(z),\sigma_\varphi^i(y))
	<s\right\}+N}{n}} \\
& = & {\displaystyle\limsup_{n\to\infty}
	\dfrac{\#\left\{i\in\{0,\ldots,n-1\}:d(\sigma_\varphi^i(z),\sigma_\varphi^i(y))
	<s\right\}}{n}} = F^*_{zy}(s)
\end{eqnarray*}
so $F^*_{xy}(r)\leq F^*_{zy}(s)$, and by a similar method $F_{xy}(r)\leq F_{zy}(s)$.
\end{proof}
\begin{corollary}\label{salam80}
For   $x=(x_\alpha)_{\alpha\in\Gamma}\in X^\Gamma$ choose $y^x=(y^x_\alpha)_{\alpha\in\Gamma}\in X^\Gamma$
such that $\{\alpha\in\Gamma:x_\alpha\neq y^x_\alpha\}$ is a finite collection of non--quasi periodic
points of $\varphi$. Moreover consider $D\subseteq X^\Gamma$ and $u,v\in X^\Gamma$, then:
\begin{itemize}
\item[1.] ``for all $t>0$ we have $F_{uv}^*(t)=1$'' if and only if
	``for all $t>0$ we have $F_{y^uy^v}^*(t)=1$'',
\item[2.] ``there exists $t>0$ such that for all distinct $x,z\in D$ we have $F_{xz}(t)=0$'' if and only if
	``there exists $t>0$ such that for all distinct $x,z\in D$ we have  $F_{y^xy^z}(t)=0$'',
\item[3.] ``there exists $t>0$ such that for all distinct $x,z\in D$ we have $F_{xz}(t)<1$'' if and only if
	``there exists $t>0$ such that for all distinct $x,z\in D$ we have  $F_{y^xy^z}(t)<1$'',
\end{itemize}
In particular for $i\in\{1,2\}$,
$D$ is distributional scrambled subset of type $i$ if and only if
$\{y^x:x\in D\}$ is distributional scrambled subset of type $i$.
\end{corollary}
\begin{proof}
{\bf 1.} Suppose for all $t>0$ we have $F_{uv}^*(t)=1$, and consider $s>0$,
then by Lemma~\ref{salam75} there exist $r_1,r_2>0$ with $F^*_{uv}(r_2)\leq F^*_{y^uv}(r_1)\leq F^*_{y^uy^v}(s)$.
Using $1=F^*_{uv}(r_2)\leq F^*_{y^uy^v}(s)\leq1$ we have $F^*_{y^uy^v}(s)=1$. Thus for
all $s>0$, $F^*_{y^uy^v}(s)=1$ is valid.
\vspace{3mm}
\\
{\bf 2.} Suppose there exists $s>0$ with $F_{xz}(s)=0$ for all distinct $x,z\in D$, then by
Lemma~\ref{salam75}, there exist $r_1,r_2>0$
such that $F_{y^xy^z}(r_2)\leq F_{y^xz}(r_1)\leq F_{xz}(s)$ for all distinct $x,y\in D$. Using $0\leq F_{y^xy^z}(r_2)\leq F_{xz}(s)=0$
we have $F_{y^xy^z}(r_2)=0$ for all distinct $x,z\in D$.
\vspace{3mm}
\\
{\bf 3.}  Use a similar method described in (2).
\end{proof}
\begin{theorem}\label{salam90}
$(X^\Gamma,\sigma_\varphi)$ is dense distributional chaotic if and only if
$\varphi:\Gamma\to\Gamma$ does not have any periodic point.
\end{theorem}
\begin{proof}
By Lemma~\ref{salam70}, if $(X^\Gamma,\sigma_\varphi)$ is dense distributional chaotic,
then $\varphi:\Gamma\to\Gamma$ does not have any periodic point.
Now suppose $\varphi:\Gamma\to\Gamma$ does not have any periodic point, then
$(X^\Gamma,\sigma_\varphi)$ is uniform distributional chaotic
by Theorem~\ref{salam60}. Consider uncountable subset $D$ of $X^\Gamma$
such that there exists $\varepsilon>0$ with 
\linebreak
$F_{xy}^*(t)=1$ and $F_{xy}(\varepsilon)=0$
for all $t>0$ and distinct $x,y\in D$. Since $\Gamma$ is countable,
\linebreak
${\mathcal P}_\textrm{fin}(\Gamma):=\{A\subseteq\Gamma:A$ is finite$\}$ is countable too,
and using finiteness of $X$, for all nonempty finite subset $A$ of
$\Gamma$, $X^A$ is finite too.
Thus 
\[\mathcal{J}:=\bigcup\{X^A:A\in{\mathcal P}_{\textrm{fin}}(\Gamma)
\setminus\{\varnothing\}\}\]
is countable. Moreover
$\Gamma$ is infinite, since $\Gamma\neq\varnothing$ and
$\varphi:\Gamma\to\Gamma$ does not have any periodic point.
So $\mathcal{J}$ is infinite countable set (note that $\# X\geq2$),
and there exists a bijection $\zeta:{\mathbb N}\to\mathcal{J}$.
Consider a one-to-one sequence $\{u_n\}_{n\geq1}$ in $D$ and
for all $n\geq1$ consider $v_n\in X^\Gamma$ with
(let $u_n=(u^n_\alpha)_{\alpha\in\Gamma}$ and
$v_n=(v^n_\alpha)_{\alpha\in\Gamma}$):
\[v^n_{\alpha}:=\left\{\begin{array}{lc}
	u^n_{\alpha} & \zeta(n)\in X^A\wedge\alpha\notin A\:, \\
	x_\alpha & \zeta(n)=(x_\beta)_{\beta\in A}\in X^A\wedge\alpha\in A\:. \end{array}\right.\]
Let $D_0:=(D\setminus\{u_n:n\geq1\})\cap\{v_n:n\geq1\}$. We have the following two claims:
\vspace{3mm}
\\
{\bf Claim 1.} $\{v_n:n\geq1\}$ and hence $D_0$ are dense subsets of $X^\Gamma$.
\vspace{3mm}
\\
{\it Proof of Claim 1.} If $x=(x_\alpha)_{\alpha\in\Gamma}\in X^\Gamma$ and $V$
is an open neighbourhood of $x$, then there exists $A\in{\mathcal P}_{\textrm{fin}}(\Gamma)
\setminus\{\varnothing\}$ such that ${\displaystyle\prod_{\alpha\in\Gamma}V_\alpha}
\subseteq V$, where $V_\alpha=\{x_\alpha\}$ for $\alpha\in A$ and $V_\alpha=X$
for $\alpha\in \Gamma\setminus A$, so $(x_\alpha)_{\alpha\in A}\in{\mathcal J}$.
Let $m:=\zeta^{-1}((x_\alpha)_{\alpha\in A})$, then
$v_m\in {\displaystyle\prod_{\alpha\in\Gamma}V_\alpha}
\subseteq V$ which completes the proof of Claim 1.
\vspace{3mm}
\\
{\bf Claim 2.} There exists $\lambda>0$ such that for all distinct $z,w\in D_0$:
\begin{itemize}
\item $\forall t>0\: \: \: \: \: F^*_{zw}(t)=1$,
\item $F_{zw}(\lambda)=0$
\end{itemize}
{\it Proof of Claim 2.} In Corollary~\ref{salam80} for $x\in X^\Gamma$
let $y^{u_n}=v_n$ and $y^x=x$ for $x\neq u_1,u_2,\ldots$, now use the fact that $\varphi$ does not have any
quasi-periodic point
\vspace{3mm}
\\
Using claims 1, 2, and uncountablity of $D$, shows that
$(X^\Gamma,\sigma_\varphi)$ is dense distributional chaotic.
\end{proof}
\section{Transitive distributional chaotic generalized shifts}
\noindent In this section we prove $(X^\Gamma,\sigma_\varphi)$ is transitive distributional chaotic if and only if
$\varphi:\Gamma\to\Gamma$ is one-to-one without any periodic point.
\begin{remark}\label{salam100}
$(X^\Gamma,\sigma_\varphi)$ is transitive if and only if
$\varphi:\Gamma\to\Gamma$ is one-to-one without any periodic point
\cite{dev, taher}.
\end{remark}
\begin{lemma}\label{salam110}
If
$\varphi:\Gamma\to\Gamma$ is one-to-one without any periodic point then
$X^\Gamma$ has an uncountable subset of transitive points like $M$ and there
exists $r>0$ such that for all distinct $x,y\in M$ we have:
\begin{itemize}
\item $\forall t>0\: \: \: \: \: F_{xy}^*(t)=1$,
\item $F_{xy}(r)=0$.
\end{itemize}
\end{lemma}
\begin{proof}
Suppose $\varphi:\Gamma\to\Gamma$ is one-to-one without any periodic point,
by Theorem~\ref{salam90},
$(X^\Gamma,\sigma_\varphi)$ is uniform distributional chaotic
so there exists $\varepsilon>0$ and uncountable subset $S$ of $X^\Gamma$ such that
for all distinct $x,y\in S$ we have $F^*_{xy}(t)=1$ (for all $t>0$) and
$F_{xy}(\varepsilon)=0$. Choose transitive point $(t_\alpha)_{\alpha\in\Gamma}$
in $X^\Gamma$, choose distinct $p,q\in X$, an strictly increasing sequence
$\{s_n\}_{n\geq1}$ such that $\dfrac{s_n}{s_1+\cdots+s_n+n(n-1)/2}>\dfrac{n-1}{n}$
for all $n\geq1$. Choose $\Lambda\subseteq\Gamma$ such that
$\bigcup\{\varphi^i(\Lambda):i\in{\mathbb Z}\}=\Gamma$ and
for all distinct $n,m\geq1$ and distinct $\alpha,\beta\in\Lambda$ we have
$\varphi^n(\alpha)\neq\varphi^m(\beta)$.
For all $\theta\in\Lambda$ and nonempty subset $A$ of
${\mathbb N}$ let:
\vspace{3mm}
\\
$(x_\theta^A,x_{\varphi(\theta)}^A,x_{\varphi^2(\theta)}^A,\cdots)$
\[=(\underbrace{z_1^A,\cdots,z_1^A}_{s_1\textrm{ times}},t_\theta,
\underbrace{z_2^A,\cdots,z_2^A}_{s_2\textrm{ times}},t_\theta,t_{\varphi(\theta)},
\underbrace{z_3^A,\cdots,z_3^A}_{s_3\textrm{ times}},t_\theta,t_{\varphi(\theta)}
,t_{\varphi^2(\theta)},\underbrace{z_1^A,\cdots,z_1^A}_{s_4\textrm{ times}},\cdots) \]
\vspace{3mm}
where:
\[z_i^A=\left\{\begin{array}{lc} p & i\in A\: , \\
	q & \textrm{ otherwise\:}. \end{array}\right.\]
Also suppose $x_\beta^A=q$ for
$\beta\in\Gamma\setminus\{\varphi^n(\theta):n\geq0,\theta\in\Lambda\}$.
Now for $x^A:=(x^A_\alpha)_{\alpha\in\Gamma}$ we have the following claims.
\vspace{3mm}
\\
{\bf Claim 1.} For nonempty subset $A$ of
${\mathbb N}$, $x^A$ is a transitive point of $\sigma_\varphi:X^\Gamma\to X^\Gamma$.
\vspace{3mm}
\\
{\it Proof of Claim 1}. Suppose $U$ is a nonempty open
neighbourhood of $(u_\alpha)_{\alpha\in\Gamma}(\in X^\Gamma)$,
there exist $\alpha_1,\ldots,\alpha_m\in\Gamma$ such that ${\displaystyle\prod_{\alpha\in\Gamma}
U_\alpha}\subseteq U$, where $U_\alpha=\{u_\alpha\}$ for $\alpha\in\{\alpha_1,\ldots,\alpha_m\}$
and $U_\alpha=X$ otherwise. There exist distinct $\theta_1,\ldots,\theta_k\in\Lambda$ and $N\geq1$
such that $\alpha_1,\ldots,\alpha_m\in\bigcup\{\varphi^i\{\theta_1,\ldots,\theta_k\}:|i|\leq N\}$.
For all $\alpha\in\bigcup\{\varphi^i\{\theta_1,\ldots,\theta_k\}:|i|\leq N\}$ let $V_\alpha=\{u_\alpha\}$
and $V_\alpha=X$ for $\alpha\notin\bigcup\{\varphi^i\{\theta_1,\ldots,\theta_k\}:|i|\leq N\}$.
Since $(t_\alpha)_{\alpha\in\Gamma}$ is a transitive point of $\sigma_\varphi$,
there exists $h\geq0$ such that $\sigma_\varphi^h((t_\alpha)_{\alpha\in\Gamma})\in
{\displaystyle\prod_{\alpha\in\Gamma}V_\alpha}$.
Moreover since the Hausdorff space $X^\Gamma$ does not have any isolated point, we may consider $h$
arbitrary large, so suppose $h>N$.
Also for all $j\in\{1,\ldots, k\}$ and $i\in\{-N,\ldots,N\}$ with $\varphi^i(\theta_j)\neq
\varnothing$, we have 
\[u_{\varphi^i(\theta_j)}=t_{\varphi^{h+i}(\theta_j)}\:.\]
Let $l=s_1+1+s_2+2+s_3+3+\cdots+s_{h+N+1}$, then:
\[\forall m\in\{0,\ldots,h+N\}\:\:\forall j\:\: (x^A_{\varphi^{l+m}(\theta_j)}=t_{\varphi^m(\theta_j)})\:.\]
Therefore for all
$j\in\{1,\ldots, k\}$ and $i\in\{-N,\ldots,N\}$ with $\varphi^i(\theta_j)\neq
\varnothing$ we have 
$x^A_{\varphi^{l+h+i}(\theta_j)}
=t_{\varphi^{h+i}(\theta_j)}=u_{\varphi^i(\theta_j)}$, which shows 
\[\forall\alpha\in\bigcup\{\varphi^i\{\theta_1,\ldots,\theta_k\}:|i|\leq N\}\:\:x^A_{\varphi^{l+h}(\alpha)}=u_\alpha\]
which leads to
$\sigma_\varphi^{l+h}(x^A)\in{\displaystyle\prod_{\alpha\in\Gamma}V_\alpha}
\subseteq {\displaystyle\prod_{\alpha\in\Gamma}U_\alpha}\subseteq U$
and $\{\sigma_\varphi^i(x^A):i\geq0\}$ is a dense subset of $X^\Gamma$, i.e.,
$x^A$ is a transitive point of $\sigma_\varphi$.
\vspace{3mm}
\\
{\bf Claim 2}. For nonempty subsets $A, B$ of
${\mathbb N}$, if $A\cap B$ is infinite, then
$F_{x^Ax^B}^*(t)=1$ for all $t>0$.
\vspace{3mm}
\\
{\it Proof of Claim 2}. Consider finite subset $E$ of $\Gamma$, there exists
$N\geq1$ with $E\subseteq\bigcup\{\varphi^i(\theta):-N\leq i\leq N,\theta\in\Lambda\}$.
For all $m\in (A\cap B)
\setminus\{0,1,\ldots,2N\}$ we have $s_m\geq m>2N$ and
for all $j,k$ with
$s_1+\cdots+s_{m-1}+m(m-1)/2+N<j<s_1+\cdots+s_m+m(m-1)/2-N$ and $0\leq k\leq 2N$ we have
$x^A_{\varphi^{(j+k-N)}(\theta)}=p=x^B_{\varphi^{(j+k-N)}(\theta)}$ for all
$\theta\in\Lambda$.
Thus for all $\alpha\in
\bigcup\{\varphi^i(\theta):-N\leq i\leq N,\theta\in\Lambda\}$
and $s_1+\cdots+s_{m-1}+m(m-1)/2+N<j<s_1+\cdots+s_m+m(m-1)/2-N$
we have
$x^A_{\varphi^j(\alpha)}=x^B_{\varphi^j(\alpha)}$, which leads to
$(\sigma_\varphi^j(x^A),\sigma_\varphi^j(x^B))\in\gamma_E$.
Therefore:
\\
$\#\{i\in\{0,\ldots,s_1+\cdots+s_m+m(m-1)/2-1\}:(\sigma_\varphi^i(x^A),\sigma_\varphi^i(x^B))
	\in \gamma_E\}
\geq$
\[s_1+\cdots+s_m+m(m-1)/2-N-(s_1+\cdots+s_{m-1}+m(m-1)/2+N)-1=s_m-2N-1\]
Hence:
\\
$\mathop{\limsup}\limits_{n\to\infty}\dfrac{\#\{i\in\{0,\ldots,n-1\}:(\sigma_\varphi^i(x^A),		
	\sigma_\varphi^i(x^B))\in \gamma_E\}}{n}$
\begin{eqnarray*}
& \geq & \mathop{\limsup}\limits_{m\to\infty, m\in A\cap B}\dfrac{\#\{i\in\{0,\ldots,s_1+\cdots+s_m+m(m-1)/2-1\}:(\sigma_\varphi^i(x^A),\sigma_\varphi^i(x^B))
	\in \gamma_E\}}{s_1+\cdots+s_m+m(m-1)/2} \\
& \geq &  \mathop{\limsup}\limits_{m\to\infty, m\in A\cap B}\dfrac{s_m-2N-1}{s_1+\cdots+s_m+m(m-1)/2} \\
& \geq &  \mathop{\limsup}\limits_{m\to\infty, m\in A\cap B}(\dfrac{m-1}{m}-\dfrac{2N+1}{s_1+\cdots+s_m+m(m-1)/2})=1
\end{eqnarray*}
Thus for all finite subset $K$ of $\Gamma$ we have
\[G^*_{xAy^A}(\gamma_D)=\mathop{\limsup}\limits_{n\to\infty}\dfrac{\#\{i\in\{0,\ldots,n-1\}:(\sigma_\varphi^i(x^A),		
	\sigma_\varphi^i(x^B))\in \gamma_K\}}{n}=1\:.\]
By Lemma~\ref{salam35} for all $t>0$ we have $F^*_{x^Ax^B}(t)=1$.
\vspace{3mm}
\\
{\bf Claim 3}. There exists $t>0$ such that
for two nonempty subsets $A, B$
of ${\mathbb N}$ if $A\setminus B$ is infinite, then
$F_{x^Ax^B}(t)=0$.
\vspace{3mm}
\\
{\it Proof of Claim 3}. Consider $\theta\in\Lambda$, there exists $t>0$ with
$\{(z,w)\in X^\Gamma\times X^\Gamma:d(z,w)<t\}\subseteq\gamma_{\{\theta\}}$.
For $m\in A\setminus B$
and $s_1+\cdots+s_{m-1}+m(m-1)/2<i<s_1+\cdots+s_m+m(m-1)/2$ we have
$p=x_{\varphi^i(\theta)}^A\neq x_{\varphi^i(\theta)}^B=q$, so
and $(\sigma_\varphi^i(x^A),\sigma_\varphi^i(x^B))\notin\gamma_{\{\theta\}}$. Using
a similar method described in the proof of Claim 2 in
Lemma~\ref{salam50} we have $F_{x^Ax^B}(t)=0$.
\\
Now we are ready to complete the proof. Using Remark~\ref{salam40}
there exists uncountable family $\mathcal K$ of infinite subsets of ${\mathbb N}\setminus
2{\mathbb N}$ such that
for all $C,D\in\mathcal{K}$, $C\cap D$ is finite (thus $C\setminus D$ is infinite for all distinct $C,D\in{\mathcal K}$).
The set $\{x^{A\cup 2{\mathbb N}}:A\in{\mathcal K}\}$ is our desired uncountable subset of $X^\Gamma$.
\end{proof}
\begin{note}\label{salam120}
If $\varphi:\Gamma\to\Gamma$ is one-to-one without any periodic point and
$x=(x_\alpha)_{\alpha\in\Gamma}$, $y=(y_\alpha)_{\alpha\in\Gamma}$ are two points of
$X^\Gamma$ such that $\{\alpha\in\Gamma:x_\alpha\neq y_\alpha\}$ is finite, then
$x$ is a transitive point of $\sigma_\varphi:X^\Gamma\to X^\Gamma$ if and only if
$y$ is a transitive point of $\sigma_\varphi:X^\Gamma\to X^\Gamma$.
\end{note}
\begin{theorem}\label{salam130}
$(X^\Gamma,\sigma_\varphi)$ is transitive distributional chaotic if and only if
$\varphi:\Gamma\to\Gamma$ is one-to-one without any periodic point.
\end{theorem}
\begin{proof}
If $(X^\Gamma,\sigma_\varphi)$ is transitive distributional chaotic, then
by Remark~\ref{salam100},
\linebreak
$\varphi:\Gamma\to\Gamma$ is one--to--one without any periodic point.
In order to complete the proof use Lemma~\ref{salam110}, Note~\ref{salam120} and
a similar method described in Theorem~\ref{salam90}.
\end{proof}
\section{Counterexamples}
\noindent Using Theorems~\ref{salam60},~\ref{salam90} and~\ref{salam130}
we have the following diagram:
\[\xymatrix{
(X^\Gamma,\sigma_\varphi)\textrm{  is  transitive  distributional  chaotic}
\ar@{=>}[d] \\
(X^\Gamma,\sigma_\varphi)\textrm{ is dense distributional chaotic}
\ar@{=>}[d] \\
(X^\Gamma,\sigma_\varphi)\textrm{ is uniform distributional 
chaotic} }
\]
Now suppose $\Gamma$ is infinite so we may suppose 
$\Gamma=\{\theta_n:n\in{\mathbb Z}\}$ with distinct $\theta_n$s.
Define $\varphi_1,\varphi_2,\varphi_3:\Gamma\to\Gamma$ with
$\varphi_1(\theta_n)=\theta_{n+1}$, $\varphi_2(\theta_n)=\theta_{n^2+1}$,
$\varphi_3(\theta_n)=\theta_{n^2}$ for $n\in{\mathbb Z}$.
Then: 
\vspace{3mm}
\\
$\bullet$ $(X^\Gamma,\sigma_{\varphi_1})$ is transitive distributional chaotic,
\vspace{3mm}
\\
$\bullet$ $(X^\Gamma,\sigma_{\varphi_2})$ is dense distributional chaotic
	and it is not transitive distributional chaotic,
\vspace{3mm}
\\
$\bullet$ $(X^\Gamma,\sigma_{\varphi_3})$ is uniform distributional chaotic
	and it is not dense distributional chaotic.
\subsubsection*{Distributional chaos and product of the generalized shifts} Consider nonempty countable sets
$\Lambda,\Upsilon$ and self--maps $\lambda:\Lambda\to\Lambda$
and $\mu:\Upsilon\to\Upsilon$, consider
\[\sigma_\lambda\times\sigma_\mu:
\mathop{X^\Lambda\times X^\Upsilon\to 
X^\Lambda\times X^\Upsilon}\limits_{\: \: \: \: \: \: \: \: \: \: \: \: \: \: \: (x,y)\mapsto(
\sigma_\lambda(x),\sigma_\mu(y))}\:,\]
then we have
(use the fact that $\sigma_\lambda\times\sigma_\mu:
X^\Lambda\times X^\Upsilon\to X^\Lambda\times X^\Upsilon$ 
is just $\sigma_{\lambda\sqcup\mu}:
X^{\Lambda\sqcup\Upsilon}\to X^{\Lambda\sqcup\Upsilon}$):
\vspace{3mm}
\\
$\bullet$ $(X^\Lambda\times X^\Upsilon,\sigma_\lambda\times\sigma_\mu)$
	is uniform distributional chaotic if and only if  
	$(X^\Lambda,\sigma_\lambda)$ or $(X^\Upsilon,\sigma_\mu)$
	is uniform distributional chaotic,
\vspace{3mm}
\\
$\bullet$ $(X^\Lambda\times X^\Upsilon,\sigma_\lambda\times\sigma_\mu)$
	is dense distributional chaotic if and only if  
	$(X^\Lambda,\sigma_\lambda)$ and $(X^\Upsilon,\sigma_\mu)$
	are dense distributional chaotic,
\vspace{3mm}
\\
$\bullet$ $(X^\Lambda\times X^\Upsilon,\sigma_\lambda\times\sigma_\mu)$
	is transitive distributional chaotic if and only if  
	$(X^\Lambda,\sigma_\lambda)$ and $(X^\Upsilon,\sigma_\mu)$
	are transitive distributional chaotic.
\subsubsection*{Distributional chaos and composition of the generalized shifts} Note that for 
\linebreak
$\eta:\Gamma\to\Gamma$ we have $\sigma_\varphi\circ\sigma_\eta=\sigma_{\eta\circ\varphi}$,  now we have:
\vspace{3mm}
\\
$\bullet$ for 
$\mathop{\lambda:{\mathbb Z}\to{\mathbb Z}}\limits_{\: \: \: \: \: \: \: \: \: \: n\mapsto n+1}$,
$(X^{\mathbb Z},\sigma_\lambda)$ and 
$(X^{\mathbb Z},\sigma_{\lambda^{-1}})$ are
transitive (dense, uniform) distributional chaotic but
$(X^{\mathbb Z},\sigma_\lambda\circ\sigma_{\lambda^{-1}})$ is not 
transitive (dense, uniform) distributional chaotic.
\vspace{3mm}
\\
$\bullet$ for 
$\lambda,\mu:{\mathbb Z}\to{\mathbb Z}$ with
\[\lambda(n):=\left\{\begin{array}{lc} n+1 & n\textrm{ \: is  even\:}, \\
n-1 & n\textrm{ \: is odd\:}, \end{array}\right.\: \: \: \: \:
\mu(n):=\left\{\begin{array}{lc} n+1 & n\textrm{ \: is odd\:}, \\
n-1 & n\textrm{ \: is even \:}, \end{array}\right.\]
then neither $(X^{\mathbb Z},\sigma_\lambda)$ nor 
$(X^{\mathbb Z},\sigma_\mu)$ are
uniform (dense, transitive) distributional chaotic but both
$(X^{\mathbb Z},\sigma_\mu\circ\sigma_\lambda)$ and
$(X^{\mathbb Z},\sigma_\lambda\circ\sigma_\mu)$ are 
uniform (dense, transitive) distributional chaotic. 

{\bf Zahra Nili Ahmadabadi}, Science and Research Branch,  Islamic Azad University, Tehran, Iran, zahra.nili@srbiau.ac.ir

{\bf Fatemah Ayatollah Zadeh Shirazi}, Faculty of Mathematics, Statistics and Computer Science,
College of Science, University of Tehran,
Enghelab Ave., Tehran, Iran, fatemah@khayam.ut.ac.ir

\end{document}